\documentclass[11pt]{article}
\usepackage[utf8]{inputenc}
\usepackage{amsmath,amsthm,amsfonts,color,cite}
\usepackage{amssymb}
\usepackage{graphicx,ulem}
\usepackage{setspace}
\usepackage{mathrsfs}
\usepackage{graphicx}
\usepackage[shortlabels]{enumitem}
\usepackage{calc}
\usepackage{tikz}
\usepackage{geometry}
\usepackage{pgffor}%可以使用foreach的包
\usepackage{ifthen}%可以使用ifthenelse的包，还能使用whiledo
\usetikzlibrary{arrows.meta}%画箭头用的包
\usetikzlibrary{decorations.markings}
\tikzstyle{vertex}=[circle, draw, inner sep=2pt, minimum size=6pt]
\tikzstyle{filledvertex}=[circle, draw, fill, inner sep=2pt, minimum size=6pt]

\usetikzlibrary{shapes.geometric}
\newtheorem{theorem}{Theorem}
\newtheorem{conjecture}[theorem]{Conjecture}
\newtheorem{lemma}[theorem]{Lemma}

\newtheorem{claim}[theorem]{Claim}
\newtheorem{corollary}[theorem]{Corollary}
\newtheorem{remark}[theorem]{Remark} 
\numberwithin{theorem}{section}
\numberwithin{equation}{section}

\begin{document}
\setstretch{1.25}
\title{Towards a strengthening of the second neighborhood conjecture}
\author{Yandong Bai
\thanks{School of Mathematics and Statistics, Xi’an--Budapest Joint Research Center for Combinatorics, Northwestern Polytechnical University, Xi'an 710129, China; Research $\&$ Development Institute of Northwestern Polytechnical University in Shenzhen, Shenzhen 518057, China. E-mail: {\tt bai@nwpu.edu.cn}.}
\qquad Binlong Li
\thanks{School of Mathematics and Statistics, Xi’an--Budapest Joint Research Center for Combinatorics, Northwestern Polytechnical University,
Xi'an 710129, China. E-mail: {\tt binlongli@nwpu.edu.cn}.}
\qquad Boram Park
\thanks{Department of Mathematics Education, Seoul National University, Seoul 08826, Republic of Korea. E-mail: {\tt borampark@snu.ac.kr}.}
} 

\maketitle

\begin{abstract}
A longstanding conjecture of Seymour, called Seymour's second neighborhood conjecture, states that every oriented graph $D$ contains a vertex $x$ with $|N^{++}_D(x)|\geq |N^{+}_D(x)|$. The conjecture was verified in a few special classes of oriented graphs, and it remains open for general oriented graphs. We study a stronger property, asking for a vertex $x$ such that there exists a complete matching from $N^+_D(x)$ to $N^{++}_D(x)$. We prove that this stronger version holds for every oriented graph with minimum out-degree at most $5$, and also for every $5$-anti-transitive oriented graph. This implies that every oriented planar graph satisfies the stronger version.
\end{abstract}

%%%%%%%%%%%%%%
\section{Introduction}
%%%%%%%%%%%%%%

A directed graph (or \textit{digraph} for short) 
has no multiple edges or loops. Let $V(D)$ and $A(D)$ be the vertex set and the arc set of a digraph $D$, respectively.
An \textit{oriented} graph is a digraph without 2-cycles.
For a cycle (resp. path) in a digraph,
we always mean a \textit{directed} cycle (resp. directed path).
The \textit{length} of a cycle or a path is the number of arcs in the cycle or path.  

For a vertex $v$ of a digraph $D$, denote by $N_D^+(v)$ and $N_D^-(v)$ the sets of out-neighbors and in-neighbors of $v$ in $D$, respectively; their cardinalities, denoted by $d_D^+(v)$ and $d_D^-(v)$, are called the \textit{out-degree} and \textit{in-degree} of $v$, respectively. We also denote by 
$\delta^+(D)$ and $\Delta^+(D)$ the minimum out-degree and the maximum out-degree of $D$, respectively.   
We call $N^{++}_D(v)$ the \textit{second out-neighborhood} of $v$, that is, 
\[N^{++}_D(v)=\{z\in V(D)\setminus (N_D^+(v)\cup\{v\}): (v,y), (y,z)\in A(D) \text{ for some }y\in N_D^+(v)\}.\] We denote by $d^{++}_D(v)=|N^{++}_D(v)|$.  
A vertex $v$ is a \textit{Seymour vertex} of $D$ if $d^{++}_D(v)\geq d^+_D(v)$. In 1990, Seymour  proposed the following conjecture, called \textit{the second neighborhood conjecture}.

\begin{conjecture}[\cite{DL95}]\label{conj:SSNC}
Every oriented graph contains a Seymour vertex.
\end{conjecture}

In a complete symmetric digraph, every two vertices form a $2$-cycle and each vertex has an empty second out-neighborhood. Thus, not every digraph contains a Seymour vertex, and the assumption that the digraph is an oriented graph cannot be omitted. The restriction of Conjecture~\ref{conj:SSNC} to tournaments was called Dean’s conjecture. It was first verified by Fisher~\cite{Fisher1996}  and  Havet and Thomass\'{e}~\cite{Havet2000} provided another proof using median orders.
The second neighborhood conjecture has received considerable attention over the years (see \cite{AGGWYZ2024, BM25, CC2023, D20, HP24, BBKS2009, DFJN2022, LX2017, DGGK2024, FY2007, G2012,  KanekoLocke2001,   Lim2020, Llado2013,H+21}). 

Kaneko and Locke~\cite{KanekoLocke2001} showed that Conjecture~\ref{conj:SSNC} holds for oriented graphs $D$ with $\delta^+(D)\le 6$. A digraph $D$ is \textit{$k$-anti-transitive} if $(v_0,v_k)\not\in A(D)$ for every path $v_0v_1\cdots v_k$ of length $k$. Let $at(D)$ be the smallest $k$ such that $D$ is $k$-anti-transitive. Daamouch~\cite{D20} proved Conjecture~\ref{conj:SSNC}  when  $at(D)\le 5$ for an oriented graph $D$, and 
Hassan, Khan, Poshni, Shabbir~\cite{H+21} proved
 the analogous result for oriented graphs $D$ with
$at(D)\le 6$. 
 Despite various results on special cases and attempts to tackle the original form, Conjecture~\ref{conj:SSNC} remains open for general oriented graphs.

A \textit{matching} of an undirected graph is a set of disjoint edges. For two disjoint vertex sets $X$ and $Y$ in an oriented graph $D$, a \textit{matching} from $X$ to $Y$ is  a set $M=\{(x_i,y_i) \in A(D) \mid x_i\in X, y_i\in Y, i=1,\ldots,m\}$ of arcs when  $\{x_iy_i \mid i=1,\ldots,m\}$ is a matching in the underlying graph. If every vertex of $X$ is a tail of some arc of a matching $M$, then it is called a \textit{complete matching} from $X$ to $Y$. 
We call a vertex $x$ a \textit{strong Seymour vertex} in an oriented graph  $D$ if there is a complete matching from $N^+_D(x)$ to $N^{++}_D(x)$. 
Clearly, a strong Seymour vertex is also a Seymour vertex,
but the converse does not hold in general; 
see, e.g., an oriented graph in Figure~\ref{fig:ex}, in which $x$ is a Seymour vertex but is not a strong Seymour vertex.

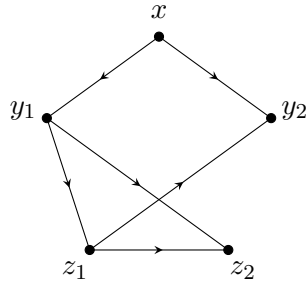
\begin{figure}[ht]
\centering
\begin{tikzpicture}
[scale=0.6, myarrow/.style={postaction={decorate, decoration={pre length=0.1cm, markings, mark=at position #1 with {\arrow{stealth}}}}}, myarrow/.default=0.5]
\foreach \x in {1,2,...,5} {\draw[fill=black] (\x*72+18:2.6) {coordinate (x\x)} circle (0.1);
\coordinate (y\x) at (\x*72+18:3.15);}
\node at (y1) {$x$}; \node at (y2) {$y_1$}; \node at (y3) {$z_1$}; 
\node at (y4) {$z_2$}; \node at (y5) {$y_2$};
\draw[myarrow] (x1)--(x2);
\draw[myarrow] (x1)--(x5);
\draw[myarrow] (x2)--(x3);
\draw[myarrow] (x2)--(x4);
\draw[myarrow] (x3)--(x4);
\draw[myarrow] (x3)--(x5);
\end{tikzpicture}
\caption{An oriented graph.}
\label{fig:ex}
\end{figure}

The following  is stronger than Conjecture~\ref{conj:SSNC}.

\begin{conjecture}\label{conj:SMC}
Every oriented graph contains a strong Seymour vertex.
\end{conjecture}

It is clear from the definition that if Conjecture~\ref{conj:SMC} is true then
Conjecture~\ref{conj:SSNC} is true. This property raises new challenges in the
case of tournaments, where the existing proofs of Conjecture~\ref{conj:SSNC} do
not apply directly: the median order method, which provides an elegant proof of
Conjecture~\ref{conj:SSNC} for tournaments, no longer works here. For this
reason we state the tournament case separately.

% It is clear from the definition that if Conjecture~\ref{conj:SMC} is true then  Conjecture~\ref{conj:SSNC} is true.
% This conjecture is interesting not only because it implies the original one, but also because it raises new challenges in the case of tournaments, where existing proofs do not apply directly and a new approach is needed. 
% Even in tournaments, it remains an open problem to find the answer to the above conjecture. The median order method for tournaments, which provides an elegant proof of Conjecture~\ref{conj:SSNC}, no longer works for Conjecture~\ref{conj:SMC}. Therefore, even the tournament case is interesting in its own right, so we state the conjecture separately.

\begin{conjecture}\label{conj:SMCinT}
Every tournament contains a strong Seymour vertex.
\end{conjecture}

After the first version of this paper appeared, David Dzitsoev was the first to communicate to us a
$36$-vertex oriented graph with no strong Seymour vertex, disproving
Conjecture~\ref{conj:SMC}. Replacing each independent set of that construction
by a transitive tournament yields a $36$-vertex tournament with the same
property, so Conjecture~\ref{conj:SMCinT} is false as well. The construction is presented in
Section~\ref{sec:tour}.

Nevertheless, a strong Seymour vertex does exist in several natural classes.
Our main results are the following.

\begin{theorem}\label{thm:SMC-MinOutdegree}
Let $D$ be an oriented graph with $\delta^+(D)\le 4$. There is a vertex with a minimum out-degree that is also a strong Seymour vertex. 
\end{theorem}

\begin{theorem}\label{thm:main} Let $D$ be an oriented graph such that  $\delta^+(D)\le 5$ or $at(D)\le 5$. Then $D$ has a strong Seymour vertex. 
\end{theorem}

The condition $\delta^+(D)\le 4$ in Theorem~\ref{thm:SMC-MinOutdegree} is essential since there is a tournament $D$ with $\delta^+(D)=5$ such that a vertex with a minimum out-degree cannot be a  strong Seymour vertex, which is discussed in Section~\ref{sec:tour}. 
Theorem~\ref{thm:main} extends the result of Daamouch~\cite{D20} to a strong Seymour vertex, where the result
was the existence of a Seymour vertex in a class of oriented graphs $D$ with $at(D)\le 5$. Moreover, since a planar graph has minimum degree at most five, the following immediately holds from Theorem~\ref{thm:main}.

\begin{corollary} 
Conjecture~\ref{conj:SMC} holds for planar oriented graphs.
\end{corollary}

The paper is organized as follows. We give proofs of Theorems~\ref{thm:SMC-MinOutdegree}~and~\ref{thm:main} in Section~\ref{sec:proof}, and then we provide in Section~\ref{sec:tour} some observations on tournaments.

\section{Proofs}\label{sec:proof}
 
\subsection{Preliminaries}

This subsection collects the definitions, known results, and simple observations that will be needed for the proofs.

Let $D$ be a digraph. For a subset $S\subset V(D)$, we denote by $D[S]$ the subgraph induced by $S$. 
For disjoint subsets $S$ and $T$ of $V(D)$, $S$ \textit{dominates} $T$ if there is an arc from every vertex of $S$ to every vertex of $T$. If $S$ or $T$ is a singleton, then we often drop the set notation, and so we say $u$ dominates $v$ if $S=\{u\}$ and $T=\{v\}$, $u$ dominates $T$ if $S=\{u\}$, or $S$ dominates $v$ if $T=\{v\}$.
If $(u,v)\in A(D)$ or $(v,u)\in A(D)$, then we say $uv$ is an edge of $D$ and $u$ is adjacent to $v$  in $D$.
The out-degree sequence $OS(D)=(d_1,\ldots,d_n)$ of $D$ is a non-increasing sequence of the out-degrees of the vertices of $D$, that is, 
$V(D)=\{v_1,\ldots,v_n\}$, $d^+_D(v_i)=d_i$ ($1\le i\le n$) and $d_1\ge d_2\ge \cdots \ge d_n$.

Presented below is the widely known Hall's Theorem.
For $X\subset V(D)$, $N_D^+(X)=\left(\cup_{x\in X} N^+_D(x)\right)\setminus X$.

\begin{theorem}[Hall's Theorem \cite{BM08}]\label{thm:Hall}
For an oriented graph $D$ with bipartition $(S,T)$, 
there exists a complete matching from $S$ to $T$ if and only if 
$|X| \le |N_D^+(X)\cap T|$ for every $X\subset S$.
\end{theorem}

\begin{lemma}\label{lemma:outdegre-sum}
For an oriented graph $D$ and two disjoint sets $S$ and $T$ of $V(D)$,
if for any $X\subseteq S$ we have $\sum_{x\in X} |N^+_{D}(x)\cap T| > |X|(|X|-1)$,
then there exists a complete matching from $S$ to $T$.
\end{lemma}

\begin{proof}
By an averaging argument, the condition implies that any subset $X$ of $S$ has at least $|X|$ out-neighbors in $T$. The statement then follows from Hall's theorem.
\end{proof}

The following is folklore.

\begin{lemma}\label{LePath4}
    Let $D$ be an oriented graph with $\delta^+(D)\geq 1$. The following statements hold:\\
    \indent {\rm(1)} $D$ contains a path of length $3$, or each component of $D$ is a triangle.\\
    \indent {\rm(2)} If $D$ has no two adjacent vertices both of out-degree $1$, then $D$ contains a path of length $4$.
\end{lemma}

\begin{proof}
    \noindent  (1) By the fact that $\delta^+(D)\geq 1$, any maximal path $P$ in a component of $D$ has length at least 2. We are done when $|E(P)|\geq 3$, so assume that $|E(P)|=2$ and say $P=v_1v_2v_3$. By $d^+_D(v_3)\geq 1$, $(v_3,v_1)\in A(D)$. By the maximality of $|E(P)|$, any vertex outside $P$ is not adjacent to $P$, implying the desired assertion.
    
    \noindent  (2) 
    Let $P$ be a longest path of $D$, say $P=v_1v_2\cdots v_{t-1}v_t$. Suppose that $D$ contains no path of length $4$, i.e., $t\leq 4$. 
    
    We claim that $(v_t,v_1)\notin A(D)$. Suppose that $(v_t,v_1)\in A(D)$. Then $C:=P\cup\{(v_t,v_1)\}$ is a cycle. We have that $V(D)=V(C)$; for otherwise together with any edge from $C$ to $V(D)\backslash V(C)$ we can get a path longer than $P$. It follows that $|V(D)|\in\{3,4\}$. But now $C$ has two adjacent vertices both of out-degree $1$, a contradiction. Thus $(v_t,v_1)\not\in A(D)$.

    Since $(v_t,v_{t-1})\notin A(D)$ and $\delta^+(D)\geq 1$, $N_D^+(v_t)\subset V(P)\backslash\{v_1,v_{t-1},v_t\}$. Then we have that $t=4$ and $N_D^+(v_t)=\{v_2\}$. In particular, we see that $d_D^+(v_t)=1$, following which we have that $d_D^+(v_2)\geq 2$. Let $u\in N_D^+(v_2)\backslash\{v_3\}$. Note that $u\notin V(P)$. Now $P'=v_3v_4v_2u$ is a longest path of $D$ as well. By a similar analysis above, we have $d_D^+(u)=1$ and $N_D^+(u)=\{v_4\}$, which contradicts the fact that $D$ has no two adjacent vertices both of out-degree $1$.
\end{proof}

\begin{lemma}[Daamouch \cite{D20}]\label{lem:at}
Let $D$ be a $k$-anti-transitive digraph. For any vertex 
$u$, if there exists a path of length 
$d$ in $D[N_D^+(u)]$, then $d\le k-2$.
\end{lemma}

The following lemma shows some observations on a vertex that is not a strong Seymour vertex. 
For a digraph $D$,  let $N_S^+(y)=N_D^+(y)\cap S$
and $d_S^+(y)=|N_S^+(y)|$, 
where  $S\subset V(D)$ and $y\in V(D)$. 

\begin{lemma}\label{lemma:observation}
Let $D$ be an oriented graph, and $x$ be a vertex of $D$ with minimum out-degree $\delta:=\delta^+(D)$. Suppose that there is $S\subset N_D^+(x)$ such that $|S|>|N^+_D(S)\setminus N_D^+(x)|$ and 
$|S'|\le |N^+_D(S')\setminus N_D^+(x)|$ for every proper subset $S'$ of $S$. Set $S^+=N_D^+(S)\backslash N_D^+(x)$ and $T=N_D^+(x)\backslash S$. (See Figure~\ref{fig1}.) The following hold. \\
\indent {\rm(1)} For every $y\in N_D^+(x)$, $N_D^+(y)\setminus N_D^+(x)\neq\emptyset$.\\
\indent {\rm(2)} $|S^+|=|S|-1$.\\
\indent {\rm(3)} $\delta^+(D[S])\geq 1$ and $|S|\geq 3$.\\
\indent {\rm(4)} If $d_S^+(y)=1$ for $y\in S$, then $d^+_D(y)=\delta$ and $y$ dominates $S^+\cup T$.
\end{lemma}

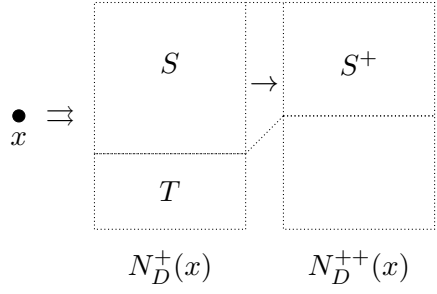
\begin{figure}[ht]
\begin{center}
\begin{tikzpicture}
\draw[fill=black] (0,0) {coordinate (x)} circle (0.08);
\node at (0,-0.3) {$x$};
\node at (0.55,0){$\rightrightarrows$};
\node at (3.25,0.4){$\rightarrow$};
\draw [densely dotted] (1.0,-1.5) rectangle (3,-0.5);
\draw [densely dotted] (1.0,-0.5) rectangle (3,1.5);
\draw [densely dotted] (3,1.5) -- (3.5,1.5);
\draw [densely dotted] (3,-0.5) -- (3.5,0);
\node at (2,0.7) {$S$};
\node at (2,-1) {$T$};
\node at (2,-2) {$N_D^+(x)$};
\draw [densely dotted] (3.5,0) rectangle (5.5,1.5);
\node at (4.5,0.7) {$S^+$};
\draw [densely dotted] (3.5,0) rectangle (5.5,-1.5);
%\node at (4.5,-0.6) { };
\node at (4.5,-2) {$N_D^{++}(x)$};
\end{tikzpicture}
\end{center}
\caption{An illustration of the sets in Lemma~\ref{lemma:observation}.}\label{fig1}
\end{figure}

\begin{proof}
\noindent (1) It follows, since $|N_D^+(y)\cap N_D^+(x)|\le \delta-1$ and $|N_D^+(y)|\ge \delta$.

\noindent (2) For any proper subset $S'$ of $S$ with size $|S|-1$, we have 
\[ |S|-1=|S'|\le |N^+_D(S')\setminus N_D^+(x)|\le |S^+|\le|S|-1,\]
and so $|S^+|=|S|-1$.

\noindent (3) For each $y\in S$, $N_D^+(y)$ is a subset of $T\cup N_S^+(y)\cup S^+$. By (2), 
\[ \delta\le d_D^+(y)\le|T|+d_S^+(y)+|S^+|=|T|+d_S^+(y)+|S|-1=d_S^+(y)+\delta-1,\]
and so $d^+_S(y)\ge 1$. It follows that $\delta^+(D[S])\geq 1$ and $|S|\geq 3$.

\noindent (4) If $d^+_S(y)=1$ for $y\in S$, then the equality of the above holds, which implies that $d_D^+(y)=\delta$ and $S^+\cup T$ is a subset of $N_D^+(y)$.
\end{proof}

\subsection{Proof of Theorem~\ref{thm:SMC-MinOutdegree}}

To show Theorem~\ref{thm:SMC-MinOutdegree}, suppose to the contrary that there is an oriented graph $D$ with $\delta^+(D)\le 4$ such that each minimum out-degree vertex is not a strong Seymour vertex. 
Let $\delta:=\delta^+(D)$. Take a vertex $x$ with  minimum out-degree. Since $x$ is not a strong Seymour vertex, there is a minimal subset $S\subset N_D^+(x)$ such that $|S|>|N^+_D(S)\setminus N_D^+(x)|$. Set $S^+:=N^+_D(S)\setminus N_D^+(x)$ and $T:=N_D^+(x)\setminus S$. Note that $|S|+|T|=\delta$.
By Lemma~\ref{lemma:observation}~(3) and the fact that $\delta\le 4$, it follows that $3\le |S|\le 4$, and thus $\delta\ge 3$.

Suppose first that $|S|=3$. 
By Lemma~\ref{lemma:observation}~(2) and~(3), 
$|S^+|=2$ and $D[S]$ is a directed triangle, say $v_1v_2v_3v_1$. By Lemma~\ref{lemma:observation}~(4), each vertex in $S$ has out-degree $\delta$ and dominates $S^+\cup T$. Note that each vertex in $S^+$ has at least one out-neighbor in $V(D)\setminus(S\cup S^+)$, and at least one vertex in $S^+$ has at least two out-neighbors in $V(D)\setminus(S\cup S^+)$. By Theorem~\ref{thm:Hall}, there is a complete matching $M$ from $S^+$ to $V(D)\setminus(S\cup S^+)$. If $\delta=3$, then $T=\emptyset$ and it follows that $v_1$ is a strong Seymour vertex of out-degree $\delta$, since $M\cup\{(v_2,v_3)\}$ is a complete matching from $N_D^+(v_1)$ to $N_D^{++}(v_1)$, a contradiction. Thus $\delta=4$. Then $|T|=1$ and for the vertex $u\in T$, $S$ dominates $u$. Therefore, $u$ has at least two out-neighbors in $V(D)\setminus(S\cup S^+)$. Moreover, since $S$ dominates $\{u\}\cup S^+$, each vertex in $S^+$ has at least two out-neighbors in $V(D)\setminus(S\cup T\cup S^+)$ and at least one vertex of $S^+$ has at least three out-neighbors in
$V(D)\setminus(S\cup T\cup S^+)$. By Theorem~\ref{thm:Hall}, we can find a complete matching $M$ from $T\cup S^+$ to $V(D)\setminus(S\cup T\cup S^+)$. It follows that $v_1$ is a strong Seymour vertex of out-degree $\delta$, since $M\cup\{(v_2,v_3)\}$ is a complete matching from  $N^+_D(v_1)$ to $N^{++}_D(v_1)$, a contradiction.

Suppose now that $|S|=4$, following which we have $\delta=4$, and by Lemma~\ref{lemma:observation}~(2), $|S^+|=3$. By Lemma~\ref{lemma:observation}~(3), we see that there are at least two vertices $v_1,v_2\in S$ such that $d_S^+(v_i)=1$, $i=1,2$. 
By Lemma~\ref{lemma:observation}~(4), $d_D^+(v_1)=d_D^+(v_2)=4$ and $\{v_1,v_2\}$ dominates $S^+$. 
Since every vertex in $S$ has at least four out-neighbors in $D$ and $D[S]$ has at most $6$ arcs, there are at least $10$ arcs from $S$ to $S^+$. Thus, there are at most $5$ arcs from $S^+$ to $S\cup S^+$.

Suppose that some vertex $z\in S^+$ has no out-neighbor in $V(D)\setminus (S\cup S^+)$. Then the out-neighbors of $z$ are the vertices in $(S^+\backslash\{z\})\cup(S\backslash\{v_1,v_2\})$. This also implies that $D[S]$ is a tournament and so the out-degree sequence of $D[S]$ is $(2,2,1,1)$. Recall that $\{v_1,v_2\}$ dominates $S^+$, each vertex in $S^+\backslash\{z\}$ has at least one out-neighbor in $V(D)\backslash(S\cup S^+)$, and one of them has at least two out-neighbors in $V(D)\backslash(S\cup S^+)$. One can check that $z$ is a strong Seymour vertex of out-degree $\delta$, since there are a complete matching from $S\backslash\{v_1,v_2\}$ to $\{v_1,v_2\}$ and a complete matching from $S^+\setminus\{z\}$ to $V(D)\setminus (S\cup S^+)$, a contradiction.

Suppose that each vertex of $S^+$ has at least one out-neighbor in $V(D)\backslash(S\cup S^+)$. Then at least two vertices of $S^+$ have at least two out-neighbors in $V(D)\backslash(S\cup S^+)$, and
at least one vertex of $S^+$ has at least three out-neighbors in $V(D)\backslash(S\cup S^+)$. By Theorem~\ref{thm:Hall}, there is a complete matching $M$ from $S^+$ to $V(D)\backslash(S\cup S^+)$, and so $M\cup\{(u_2,u_3)\}$ is a complete matching from $N^+_D(v_1)$ to $N^{++}_D(v_1)$, 
where $u_2$ is the out-neighbor of $v_1$ in $D[S]$ and $u_3$ is an out-neighbor of $u_2$ in $D[S]$.
Thus, $v_1$ is a strong Seymour vertex of out-degree $\delta$, a contradiction.

\subsection{Proof of Theorem~\ref{thm:main}}

To show Theorem~\ref{thm:main}, suppose to the contrary that there is an oriented graph $D$ without a strong Seymour vertex such that $\delta^+(D)\le 5$ or $at(D)\le 5$. We choose such a digraph $D$ with minimum order, and then with minimum arc number. By Theorem~\ref{thm:SMC-MinOutdegree}, either $\delta^+(D)=5$, or $\delta^+(D)\geq 6$ and $at(D)\le 5$. 

\begin{claim}\label{Clxyz}
    Let $x,y,z\in V(D)$ be such that $(x,y),(y,z)\in A(D)$ and $(x,z)\notin A(D)$. There is at least one arc from $N_D^+(x)\backslash\{y\}$ to $\{y,z\}$.
\end{claim}

\begin{proof}
    Suppose that there is no arc from $N_D^+(x)\backslash\{y\}$ to $\{y,z\}$. Let $D'$ be the digraph obtained from $D$ by removing the arc $(x,y)$. By the choice of $D$, $D'$ has a strong Seymour vertex. Suppose that $u$ is a strong Seymour vertex of $D'$ with $u\neq x$, say $M$ is a complete matching from $N_{D'}^+(u)$ to $N_{D'}^{++}(u)$. Note that $N_D^+(u)=N_{D'}^+(u)$ and $N_D^{++}(u)\subseteq N_{D'}^{++}(u)$. It follows that $u$ is a strong Seymour vertex of $D$ as well, a contradiction. So we conclude that $x$ is a strong Seymour vertex of $D'$, and let $M$ be a complete matching from $N_{D'}^+(x)$ to $N_{D'}^{++}(x)$. Recall that there are no arcs from $N_D^+(x)\backslash\{y\}$ to $\{y,z\}$. This implies $y,z\notin N_{D'}^{++}(x)$. Now $M\cup\{(y,z)\}$ is a complete matching from $N_D^+(x)$ to $N_D^{++}(x)$, implying that $x$ is a strong Seymour vertex of $D$, a contradiction.
\end{proof}
 
For simplicity, we let 
$$\alpha(x):=\max\{d_{N_D^+(x)}^+(y): d_D^+(y) \text{ is minimum among the vertices in }N_D^+(x)\}$$ for any vertex $x\in V(D)$. Note that $\alpha(x)\le d^+_D(x)-1$ for every vertex $x\in V(D)$.

For any vertex $x\in V(D)$, by the fact that $x$ is not a strong Seymour vertex of $D$, and by Theorem~\ref{thm:Hall}, we see that there is a set $S\subseteq N_D^+(x)$ such that $|S|>|N^+_D(S)\setminus N_D^+(x)|$, and we denote by $\beta(x)$ the size of a smallest such set $S$. 

Let $\delta:=\delta^+(D)$. We choose a vertex $x_0$ in $D$ with $d_D^+(x_0)=\delta$ such that:\\
\indent (1) $\alpha (x_0)$ is maximum; and subject to this,\\
\indent (2) $\beta (x_0)$ is minimum; and subject to this,\\
\indent (3) The out-degree sequence of $D[N_D^+(x_0)]$ is lexicographically maximum.

Let $S\subseteq N_D^+(x_0)$ be such that $|S|>|N^+_D(S)\setminus N_D^+(x_0)|$ with $|S|=\beta(x_0)$.  For simplicity, we set $R=N_D^+(x_0)$, $T=R\setminus S$ and $S^+=N_D^+(S)\backslash R$. Let $y_0$ be a vertex in $R$ with $d_D^+(y_0)=\min\{d_D^+(y): y\in R\}$ and $d_R^+(y_0)=\alpha(x_0)$.
Note that $|R|=\delta$. By Lemma~\ref{lemma:observation}~(2), $|S^+|=|S|-1$ and $|S^+\cup T|=\delta-1$.
For a positive integer $\ell$, let $S_\ell=\{v\in S: d^+_S(v)=\ell\}$.
Note that $\delta^+(D[S])\geq 1$ by Lemma~\ref{lemma:observation}~(3).

\begin{claim}\label{ClAdjacentinS1}
  Suppose that there are two vertices in $S_1$ that are adjacent in $D$. 
  Then $|S|=4$, $y_0\in S$, $y_0$ dominates $S\backslash\{y_0\}$ and $D[S\backslash\{y_0\}]$ is a triangle.
\end{claim}

\begin{proof}
    Let $v_1,v_2\in S_1$ with $(v_1,v_2)\in A(D)$. By Lemma~\ref{lemma:observation}~(4), $d_D^+(v_1)=d_D^+(v_2)=\delta$ and $\{v_1,v_2\}$ dominates $T\cup S^+$. Clearly, by the existence of $v_2$, $\alpha(v_1)=\delta-1$, following which we have $\alpha(x_0)=\delta-1$ by the choice of $x_0$. Thus, $y_0$ has exactly one out-neighbor in $N^{++}_D(x_0)$. Recall that $|S|\geq 3$, $|S^+|\geq 2$ and $\{v_1,v_2\}$ dominates $S^+$. This implies that $y_0,v_1,v_2$ are three distinct vertices. Since $\{v_1,v_2\}$ dominates $T$ and $y_0$ dominates $R\backslash\{y_0\}$, we have that $y_0\notin T$, i.e., $y_0\in S$. Set $S'=S\backslash\{y_0\}$. 
    
   By Lemma~\ref{lemma:observation}~(3), $\delta^+(D[S])\geq 1$. 
   Since $y_0$ dominates $S'$, we have that $\delta^+(D[S'])\geq 1$. By Lemma~\ref{LePath4}, $D[S']$ is either a triangle or contains a path of length $3$. If $D[S']$ is a triangle, then $|S|=4$ and $y_0$ dominates $S'$, as desired. 
{Suppose to the contrary that $D[S']$ contains a path of length $3$. Together with an arc from $y_0$ to $S'$ we get a path of length $4$ in $D[S]$. 
Since $x_0$ dominates all vertices of this path, we have $at(D)\neq 5$ by Lemma~\ref{lem:at}. By our assumption, $\delta=5$.  Then $\beta(x_0)=|S|=5$, $S=R$, $d_D^+(y_0)=5$ and $\alpha(v_1)=4$. By the choice of $x_0$, we have that $\beta(v_1)=5$. That is, $|N_D^{++}(v_1)|<|N_D^+(v_1)|$.
Note that $OS(D[N_D^+(v_1)])=(4,b_2,b_3,b_4,b_5)$ and for each $j$, $b_j\ge 1$ by Lemma~\ref{lemma:observation}~(3). Since $b_2+b_3+b_4+b_5\le 6$, $b_4\le 1$.  
Thus, there are at least two vertices $u,u'\in  N_D^+(v_1)$ such that $d_{N_D^+(v_1)}(u)=d_{N_D^+(v_1)}(u')=1$. Note that $u,u'\neq v_2$ and so $u,u'\in S^+$. Let $v_3\in S$ be the out-neighbor of $v_2$ in $D[S]$.  By Lemma~\ref{lemma:observation}~(4), each of $u$ and $u'$ dominates $v_3$. Then $N_D^+(v_3)$ is a subset of $(S\setminus\{y_0,v_2,v_3\}) \cup (S^+\setminus\{u,u'\})$ and so 
  $|N_D^+(v_3)|\le 4$, which contradicts the fact that $\delta=5$.}
\end{proof}

We will show that $D[R]$ contains a path
of length $4$.
If any two vertices in $S_1$ are nonadjacent, then by Lemma~\ref{LePath4}~(2), $D[S]$ contains a path of length $4$, and so does $D[R]$. Now suppose that there are two vertices in $S_1$ that are adjacent. By Claim~\ref{ClAdjacentinS1}, $|S|=4$, $y_0\in S$, $D[S\backslash\{y_0\}]$ is a triangle and $y_0$ dominates $S\backslash\{y_0\}$.  We have that $S_1=S\backslash\{y_0\}$, and $T\neq\emptyset$. Let $D[S_1]=v_1v_2v_3v_1$ and $u\in T$. By Lemma~\ref{lemma:observation}~(4), $S_1$ dominates $T$. Thus, $y_0v_1v_2v_3u$ is a path of length $4$ in $D[R]$. Therefore, in any case, $D[R]$ contains a path of length $4$, and so $at(D)\neq 5$ by Lemma~\ref{lem:at}. By our assumption, $\delta=5$.

\begin{claim}\label{ClDSPath4} 
    No two vertices in $S_1$  are adjacent in $D$, and $D[S]$ contains a path of length $4$. 
\end{claim}

\begin{proof}
Suppose that there are two vertices in $S_1$ that are adjacent in $D$. 
By Claim~\ref{ClAdjacentinS1}, $|S|=4$, $y_0\in S$, $D[S\backslash\{y_0\}]$ is a triangle and $y_0$ dominates $S\backslash\{y_0\}$.
Recall that $\delta=5$, implying that $|T|=1$. We can assume that $T=\{u\}$, $S=\{y_0,v_1,v_2,v_3\}$, where $v_1v_2v_3v_1$ is a triangle. By Lemma~\ref{lemma:observation}~(2)~and~(4), $|S^+|=3$, and $\{v_1,v_2,v_3\}$ dominates $S^+\cup T$. Thus $N_D^+(v_1)=\{v_2,u\}\cup S^+$. (See Figure~\ref{fig:proof_Claim28}.)  

\begin{figure}[ht]
\begin{center}
\begin{tikzpicture}
[ scale=1.2, myarrow/.style={postaction={decorate, decoration={pre length=0.1cm, markings, mark=at position #1 with {\arrow{stealth}}}}}, myarrow/.default=0.5]
\draw[fill=black] (0,0) {coordinate (x)} circle (0.07);
\draw[fill=black] (2,-1) {coordinate (u)} circle (0.07);
\node at (2.3,-1) {$u$};
\node at (0,-0.3) {$x_0$};
\draw[fill=black] (2,1.6) {coordinate (y)} circle (0.07);
\node at (2.3,1.6) {$y_0$};
\node at (2.3,1) {$v_1$};
\node at (2.3,0.4) {$v_2$};
\node at (2.3,-0.2) {$v_3$};
\draw[fill=black] (2,1) {coordinate (v1)} circle (0.07);
\draw[fill=black] (2,0.4) {coordinate (v2)}circle (0.07);
\draw[fill=black] (2,-0.2) {coordinate (v3)} circle (0.07);
\draw[myarrow] (y) to[out=200, in=160]  (u);
\draw[myarrow] (y) to[out=220, in=140]  (v3);
\draw[myarrow] (y) to[out=240, in=120]  (v2);
\draw[myarrow] (y) -- (v1);
\draw[myarrow] (v1) -- (v2);
\draw[myarrow] (v2) -- (v3);
\draw[myarrow] (v3) to[out=120, in=240]  (v1);
\draw[myarrow] (v1) to[out=220, in=140]  (u);
\draw[myarrow] (v2) to[out=240, in=120]  (u);
\draw[myarrow] (v3) --  (u);
\node at (0.55,0){$\rightrightarrows$};
\node at (3.25,0.7){$\rightarrow$};
\draw [densely dotted] (1.0,-1.5) rectangle (3,-0.5);
\draw [densely dotted] (1.0,-0.5) rectangle (3,2);
\draw [densely dotted] (3,2) -- (3.5,2);
\draw [densely dotted] (3,-0.5) -- (3.5,0);
\node at (0.7,1.5) {$S$};
\node at (0.7,-1.2) {$T$};
\node at (2,-2) {$N_D^+(x)$};
\draw [densely dotted] (3.5,0) rectangle (5.5,2);
\draw[fill=black] (4.5,1.6) circle (0.07);
\draw[fill=black] (4.5,1) circle (0.07);
\draw[fill=black] (4.5,0.4) circle (0.07);
\node at (5.8,1.5) {$S^+$};
\draw [densely dotted] (3.5,0) rectangle (5.5,-1.5);
\node at (4.5,-2) {$N_D^{++}(x)$};
\end{tikzpicture}
\end{center}
\caption{An illustration for a proof of Claim~\ref{ClDSPath4}.}\label{fig:proof_Claim28}
\end{figure}
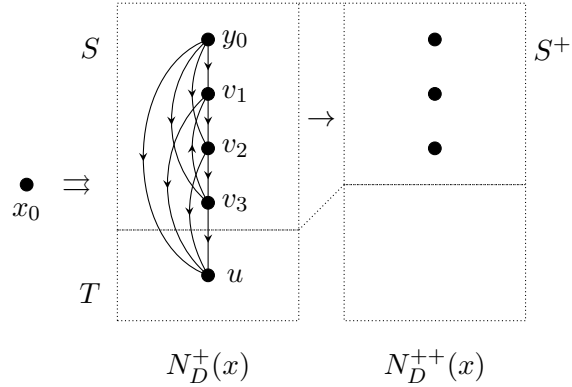

Recall that $d_D^+(v_1)=\delta$, $\alpha(v_1)=\alpha(x_0)=4$. Let $S'$ be a subset of $N_D^+(v_1)$ such that $|S'|>|N_D^+(S')\backslash N_D^+(v_1)|$ and $|S'|=\beta(v_1)$. Set $T'=N_D^+(v_1)\backslash S'$ and $S'^+=N_D^+(S')\backslash N_D^+(v_1)$. We claim that $v_2\notin S'$. Suppose that $v_2\in S'$, then $v_3\in S'^+$. Note that $v_3$ dominates $S^+\cup\{u\}$ and so $v_3$ also dominates
$S'\backslash\{v_2\}$. It follows that $|S'\backslash\{v_2\}|>|N_D^+(S'\backslash\{v_2\})\backslash N_D^+(v_1)|$, contradicting the minimality of $S'$. Thus, as we claimed, $v_2\notin S'$ and then $|S'|\leq 4$. This also implies that $S'$ contains one vertex $w$ with $d_{S'}^+(w)=1$. By Lemma~\ref{lemma:observation}~(4), $w$ dominates $T'$, contradicting that $v_2\in T'$ dominates $N_D^+(v_1)\backslash\{v_2\}$.

This shows that $S_1$ does not contain two adjacent vertices in $D$. By Lemma~\ref{LePath4}~(2), $D[S]$ contains a path of length $4$.
\end{proof}

By Claim~\ref{ClDSPath4}, we see that $|S|=5$, that is, $S=R=N_D^+(x_0)$ and $T=\emptyset$. If $D[S]$ is not strongly connected, then a strongly connected component $H$ of $D[S]$ such that there is no arc from $V(H)$ to $S\backslash V(H)$ will have minimum out-degree at least $1$ and each two vertices in $H$ with out-degree exactly $1$ are nonadjacent. By Lemma~\ref{LePath4}~(2), $H$ will contain a path of length $4$, a contradiction. So we conclude that $D[S]$ is strongly connected.
By Claim~\ref{ClDSPath4} again, no two vertices in $S_1$ are adjacent, implying that $S_4=\emptyset$ and $|S_3|\leq 1$. 

\begin{claim}\label{S4S3}
 We have $|S_3|=1$ and $|S_1|,|S_2|\geq 1$. 
\end{claim}
\begin{proof} 
From Claim~\ref{ClDSPath4},  it is sufficient to show that $S_3\neq \emptyset$. Suppose to the contrary that $S_3=\emptyset$. Then $\alpha(x_0)\le 2$.

First, assume that $S_1=\emptyset$. Then $OS(D[S])=(2,2,2,2,2)$, and $D[S]$ is a 2-regular tournament, say $S=\{v_1,v_2,v_3,v_4,v_5\}$, where $v_i$ dominates $v_{i+1}$ and $v_{i+2}$ for $i=1,\ldots,5$ (the subscripts are taken modulo 5). If there are at least two vertices in $S$ with out-degree 5 in $D$, say $d_D^+(v_1)=d_D^+(v_2)=5$, then $v_1,v_2$ have at least two common out-neighbors in $S^+$ and one common out-neighbor in $S$. It follows that $\alpha(v_1)=3>\alpha(x_0)=2$, a contradiction to the choice of $x_0$. So we can suppose that all vertices in $S\backslash\{v_1\}$ have out-degree 6 in $D$. Then $S\backslash\{v_1\}$ dominates $S^+$. Let $D'$ be the digraph obtained from $D$ by removing two arcs $(v_1,v_2)$, $(v_1,v_3)$. Then $v_1$ is the unique vertex with out-degree at most 4 in $D'$. By Theorem~\ref{thm:SMC-MinOutdegree}, $v_1$ is a strong Seymour vertex in $D'$. Note that $N_{D'}^+(v_1)\subseteq S^+$ and $N_{D'}^{++}(v_1)\cap S=\emptyset$. Let $M$ be a complete matching from $N_{D'}^+(v_1)$ to $N_{D'}^{++}(v_1)$. Then $M\cup\{(v_2,v_4),(v_3,v_5)\}$ is a complete matching from $N_D^+(v_1)$ to $N_D^{++}(v_1)$, implying that $v_1$ is a strong Seymour vertex in $D$, a contradiction. 

Second, we assume that $S_1\neq\emptyset$, say $u\in S_1$. We claim that each vertex in $N_S^+(u)\cup N_S^-(u)$ dominates $S^+$. Suppose otherwise there is a vertex $v\in N_S^+(u)\cup N_S^-(u)$ that does not dominate $S^+$. Then $d_D^+(v) \le  2+ d_{S^+}(v)\le 2+3=5$. Thus $d_D^+(v)=5$ and $u,v$ have at least three common out-neighbors in $D$. It follows that either $\alpha(u)$ or $\alpha(v)$ is greater than or equal to $3$, a contradiction. Thus, as we claimed, each vertex in $N_S^+(u)\cup N_S^-(u)$ dominates $S^+$. 

Now let $v$ be the out-neighbor of $u$ in $S$ (then $d^+_S(v)=2$ by Claim~\ref{ClDSPath4}) and let $w_1,w_2$ be the out-neighbors of $v$ in $S$. By Claim~\ref{Clxyz}, $w_1$, and similarly $w_2$, does not dominate $S^+$. It follows that $w_1,w_2\in S_2$ and $d^+_D(w_1)=d^+_D(w_2)=5$. This also implies $w_1,w_2\notin N_D^-(u)$. Thus, $d_S^+(w_1)+d_S^+(w_2)\leq 3$,  a contradiction to the fact that  $w_1,w_2\in S_2$.
\end{proof}

\begin{claim}\label{Clu1S1u2S2}
    If $u\in S_1, v\in S_2$ with $(u,v)\in A(D)$, then $d_D^+(v)=5$.
\end{claim}

\begin{proof}
    Suppose that $d_D^+(v)\neq 5$. Then $d_D^+(v)=6$ and $\{u,v\}$ dominates $S^+$. Set $N_S^+(v)=\{w_1,w_2\}$. By Claim~\ref{Clxyz}, there is at least one arc from $S^+$ to $w_1$ and at least one arc from $S^+$ to $w_2$. Then each of $w_1$ and $w_2$ does not dominate $S^+$, and therefore, $w_1,w_2\in S_2\cup S_3$ by Lemma~\ref{lemma:observation}~(4). Let $z$ be the vertex in $S\backslash\{u,v,w_1,w_2\}$.

    First assume that both $w_1$ and $w_2\in S_2$. 
    Then, since $w_i$ does not dominate $S^+$, $d_D^+(w_i)=5$ for each $i\in\{1,2\}$. In addition,
    $z\in S_3$ by Claim~\ref{S4S3}. Then the out-degree sequence of $D[N^+_D(x_0)]$ is $(3,2,2,2,1)$. 
    Moreover, this implies that $w_1,w_2$ are adjacent, say $(w_1,w_2)\in A(D)$. It follows that $(w_1,u),(w_2,z),(w_2,u),(z,u),(z,v),(z,w_1)\in A(D)$ (see Figure~\ref{FiConstructionDS}(i)). Now considering the neighbors of $w_1$, $w_2$, $u$ in $S^+$, it holds that 
    $w_1$, $w_2$ have three common out-neighbors in $D$ and $w_1$, $u$ have three common out-neighbors in $D$. It follows that $\alpha(w_1)=3$, $\beta(w_1)\leq 5$ and the out-degree sequence of $D[N^+_D(w_1)]$ is larger than that of $D[N_D^+(x_0)]$, a contradiction to the choice of $x_0$.

    Second we assume without loss of generality that $w_1\in S_3$, implying that $w_2\in S_2$ and $(w_1,w_2),(w_1,z),$ $(w_1,u),(w_2,z),(w_2,u)\in A(D)$. Moreover, we have $(z,v)\in A(D)$; for otherwise, $z\in S_1$ and $(z,u)\in A(D)$, contradicting Claim~\ref{ClDSPath4}. Note that $d_D^+(w_2)=5$ and $w_2,u$ have three common out-neighbors in $S^+$. 
    Then the out-degree sequence of $D[N^+_D(x_0)]$ is $(3,2,2,1,1)$ or $(3,2,2,2,1)$. 
    If $z\in S_1$, then $w_2,z$ have three common out-neighbors in $S^+$
    (see Figure~\ref{FiConstructionDS}~(ii)); and if $z\in S_2$, then $w_2,z$ have one common out-neighbor in $S$ and at least two common out-neighbors in $S^+$ (see Figure~\ref{FiConstructionDS}~(iii)). For each case, we have that  $\alpha(w_2)=3$, $\beta(w_2)\leq 5$ and the out-degree sequence of $D[N^+_D(w_2)]$  is larger than that of $D[N_D^+(x_0)]$, a contradiction. 
\end{proof}

\begin{figure}[ht]
\centering
\begin{tikzpicture}[scale=0.7, myarrow/.style={postaction={decorate,
    decoration={pre length=0.1cm, markings, mark=at position #1 with {\arrow{stealth}}}}}, myarrow/.default=0.5]
\begin{scope}[xshift=-10cm]
\foreach \x in {1,2,...,5} {\draw[fill=black] (\x*72+18:2.6) {coordinate (x\x)} circle (0.13);
\coordinate (y\x) at (\x*72+18:3);}
\draw[myarrow] (x1)--(x2); \draw[myarrow] (x2)--(x3); \draw[myarrow] (x2)--(x4);
\draw[myarrow] (x3)--(x4); \draw[myarrow] (x3)--(x1); \draw[myarrow] (x4)--(x5);
\draw[myarrow] (x4)--(x1); \draw[myarrow] (x5)--(x1); \draw[myarrow] (x5)--(x2); 
\draw[myarrow] (x5)--(x3);
\node at (y1) {$u$}; \node at (y2) {$v$}; \node at (y3) {$w_1$}; 
\node at (y4) {$w_2$}; \node at (y5) {$z$};
\node at (0,-3.5) {(i) $w_1,w_2\in S_2$};
\end{scope}
\begin{scope}
\foreach \x in {1,2,...,5} {\draw[fill=black] (\x*72+18:2.6) {coordinate (x\x)} circle (0.13);
\coordinate (y\x) at (\x*72+18:3);}
\draw[myarrow] (x1)--(x2); \draw[myarrow] (x2)--(x3); \draw[myarrow] (x2)--(x4);
\draw[myarrow] (x3)--(x4); \draw[myarrow] (x3)--(x5); \draw[myarrow] (x3)--(x1);
\draw[myarrow] (x4)--(x5); \draw[myarrow] (x4)--(x1); \draw[myarrow] (x5)--(x2);
\node at (y1) {$u$}; \node at (y2) {$v$}; \node at (y3) {$w_1$}; 
\node at (y4) {$w_2$}; \node at (y5) {$z$};
\node at (0,-3.5) {(ii) $w_1\in S_3$, $z\in S_1$};
\end{scope}
\begin{scope}[xshift=10cm]
\foreach \x in {1,2,...,5} {\draw[fill=black] (\x*72+18:2.6) {coordinate (x\x)} circle (0.13);
\coordinate (y\x) at (\x*72+18:3);}
\draw[myarrow] (x1)--(x2); \draw[myarrow] (x2)--(x3); \draw[myarrow] (x2)--(x4);
\draw[myarrow] (x3)--(x4); \draw[myarrow] (x3)--(x5); \draw[myarrow] (x3)--(x1);
\draw[myarrow] (x4)--(x5); \draw[myarrow] (x4)--(x1); \draw[myarrow] (x5)--(x1); 
\draw[myarrow] (x5)--(x2);
\node at (y1) {$u$}; \node at (y2) {$v$}; \node at (y3) {$w_1$}; 
\node at (y4) {$w_2$}; \node at (y5) {$z$};
\node at (0,-3.5) {(iii) $w_1\in S_3$, $z\in S_2$};
\end{scope}
\end{tikzpicture}
\caption{Configurations of $D[S]$ in Claim~\ref{Clu1S1u2S2}.}\label{FiConstructionDS}
\end{figure}
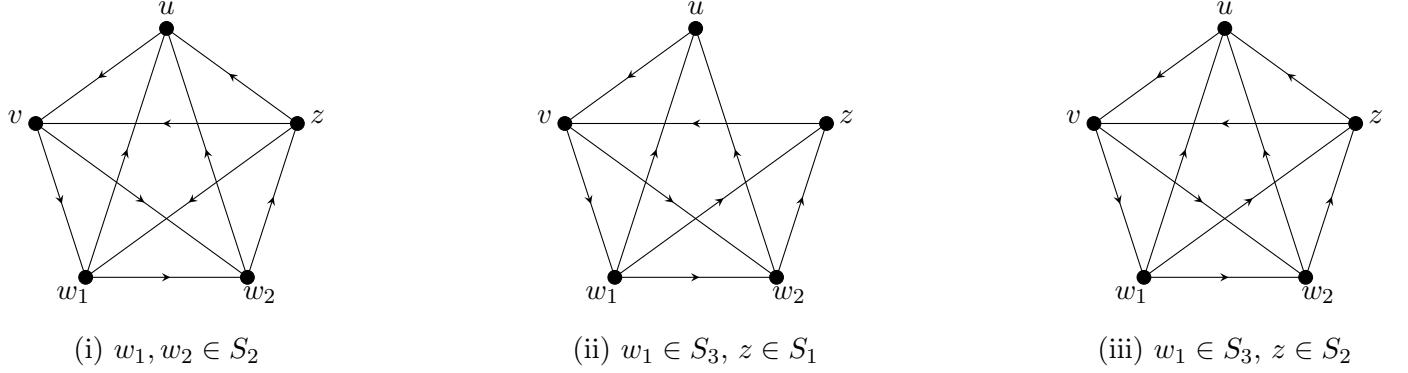

\begin{claim}\label{ClS2OutDegree5}
    There is a vertex in $S_2$ with out-degree $5$ in $D$.
\end{claim}

\begin{proof}
    Let $u\in S_1$, and $v$ be the out-neighbor of $u$ in $S$. If $v\in S_2$, then by Claim~\ref{Clu1S1u2S2}, $d_D^+(v)=5$, as desired. So assume that $v\in S_3$. Set $N_D^+(v)=\{w_1,w_2,w_3\}$. If $w_1\in S_1$, then the out-neighbor of $w_1$ in $S$ is either $w_2$ or $w_3$, which is in $S_2$ by Claims~\ref{ClDSPath4}~and~\ref{S4S3}. It follows that either $w_2$ or $w_3$ has out-degree 5 in $D$ by Claim~\ref{Clu1S1u2S2}. So we assume that $w_1$, and similarly, $w_2,w_3$, are all in $S_2$. This implies that $\{w_1,w_2,w_3\}$ dominates $u$ and $D[\{w_1,w_2,w_3\}]$ is a triangle, see Figure~\ref{FiClS2OutDegree5}.

\begin{figure}[ht]
\centering
\begin{tikzpicture}[scale=0.7, myarrow/.style={postaction={decorate,
    decoration={pre length=0.1cm, markings, mark=at position #1 with {\arrow{stealth}}}}}, myarrow/.default=0.5]
\begin{scope}%[xshift=-8cm]
\foreach \x in {1,2,...,5} {\draw[fill=black] (\x*72+18:2.6) {coordinate (x\x)} circle (0.13);
\coordinate (y\x) at (\x*72+18:3);}
\draw[myarrow] (x1)--(x2); \draw[myarrow] (x2)--(x3); \draw[myarrow] (x2)--(x4);
\draw[myarrow] (x2)--(x5); \draw[myarrow] (x3)--(x4); \draw[myarrow] (x4)--(x5);
\draw[myarrow] (x5)--(x3); \draw[myarrow] (x3)--(x1); \draw[myarrow] (x4)--(x1); 
\draw[myarrow] (x5)--(x1);
\node at (y1) {$u$}; \node at (y2) {$v$}; \node at (y3) {$w_1$}; 
\node at (y4) {$w_2$}; \node at (y5) {$w_3$};
\end{scope}
\end{tikzpicture}
\caption{Configuration of $D[S]$ in Claim~\ref{ClS2OutDegree5}.}\label{FiClS2OutDegree5}
\end{figure}
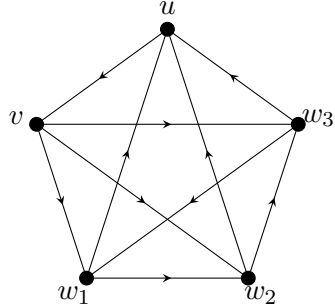

Suppose that $d_D^+(w_i)=6$ for each $i\in\{ 1,2,3\}$. Then $\{w_1,w_2,w_3\}$ dominates $S^+$. It follows that $(w_1,w_2),(w_2,w_3)\in A(D)$ and there are no arcs from $N_D^+(w_1)\backslash\{w_2\}$ to $\{w_2,w_3\}$, contradicting Claim~\ref{Clxyz}.    
\end{proof}

\begin{claim}\label{ClS2OutDegree5Seymour}
    There is a vertex in $S_2$ with out-degree $5$ in $D$ that is a strong Seymour vertex in $D[S]$.
\end{claim}

\begin{proof}
By Claim~\ref{ClS2OutDegree5},   there is $u\in S_2$ with $d_D^+(u)=5$. Suppose that $u$ is not a strong Seymour vertex in $D[S]$. 
Let $N_S^+(u)=\{v_1,v_2\}$. 
Suppose that $v_i\in S_3$ for some $i\in\{1,2\}$. We may assume that $v_1\in S_3$. 
Then $(v_1,v_2)\in A(D)$. Note that $v_2$ has at least one out-neighbor $v_3$ in $S\backslash\{u,v_1,v_2\}$ and $v_1$ dominates all vertices in $S\backslash\{u,v_1\}$, implying that $u$ is a strong Seymour vertex in $D[S]$, a contradiction. Thus, $v_1,v_2\in S_1\cup S_2$.

Suppose that $v_1,v_2\in S_2$. Then each vertex in $\{v_1,v_2\}$ has an out-neighbor in $S\backslash\{u,v_1,v_2\}$, and one vertex in $\{v_1,v_2\}$ has at least two out-neighbors in $S\backslash\{u,v_1,v_2\}$, implying that $u$ is a strong Seymour vertex in $D[S]$, a contradiction. 

Now we assume without loss of generality that $v_1\in S_1$. 
Let $w$ be the out-neighbor of $v_1$ in $D[S]$. 
If $w=v_2$, that is, $(v_1,v_2)\in A(D)$, then $v_2\in S_2$ and $d_D^+(v_2)=5$ by Claim~\ref{Clu1S1u2S2}. By Claim~\ref{S4S3}, $S_3\neq\emptyset$ and so $v_2$ has an out-neighbor in $S_3$. This implies that $v_2$ is a strong Seymour vertex in $D[S]$, as desired. So suppose that $w\neq v_2$. 
By Claims~\ref{ClDSPath4}~and~\ref{S4S3}, 
$w\in S_2\cup S_3$. 
Note that $N_S^+(v_2)\subset \{v_1,w\}$, since $u$ is not a strong Seymour vertex in $D[S]$. 
If $w\in S_3$, then 
$N_S^+(v_2)= \{v_1\}$ and so $v_2\in S_1$, a contradiction by Claim~\ref{ClDSPath4}.
Thus,  $w\in S_2$. 
Then  $d_D^+(w)=5$ by Claim~\ref{Clu1S1u2S2}. 
Together with Claim~\ref{S4S3}, one can observe that $w$ is a strong Seymour vertex in $D[S]$, as desired.  
\end{proof}

By Claim~\ref{ClS2OutDegree5Seymour}, let $u\in S_2$ be a vertex with $d_D^+(u)=5$ which is a strong Seymour vertex in $D[S]$. We set $S_{\ast}=N_{S^+}^+(u)$ and $S_{\ast}^+=N_D^+(S_{\ast})\backslash(S\cup S_{\ast})$. Note that $|S_{\ast}|=3$.

We remark that every vertex in $S_1$ has three out-neighbors in $S_{\ast}$, every vertex in $S_2$ has at least two out-neighbors in $S_{\ast}$, and so the vertex in $S_3$ has at least one out-neighbor in $S_{\ast}$. It follows that there are at least $11$ arcs from $S$ to $S_{\ast}$, and at most $4$ arcs from $S_{\ast}$ to $S$. Moreover, for every vertex $v\in S_{\ast}$, $v$ has at most $4$ out-neighbors in $S\cup S_{\ast}$, implying that $v$ has at least one out-neighbor in $S_{\ast}^+$. For a subset $X$ of $S_{\ast}$ with $|X|\geq 2$,  
$$\sum_{x\in X}d_{S_{\ast}^+}^+(x)\geq 5|X|-4-|X|(3-|X|)-{|X|\choose 2}\geq|X|(|X|-1).$$
By Lemma~\ref{lemma:outdegre-sum}, there is a complete matching from $S_{\ast}$ to $S_{\ast}^+$. Together with a complete matching from $N_S^+(u)$ to $N_S^{++}(u)$, we get a complete matching $M$ from $N_D^+(u)$ to $N_D^{++}(u)$. This implies that $u$ is a strong Seymour vertex in $D$, a contradiction. The proof is complete.

%%%%%%%%%%%%
\section{Remarks on tournaments without a strong Seymour vertex}\label{sec:tour}
%%%%%%%%%%%%

Denote by $TT_k$ a transitive tournament with $k$ vertices, where 
a \textit{transitive} digraph $D$ means that 
$(x,y), (y,z)\in A(D)$ implies $(x,z)\in A(D)$.

\begin{figure}[ht]
\begin{center}

\begin{tikzpicture}[
  scale=0.8,
  x=1cm,
  y=1cm,
  blk/.style={
    ellipse,
    draw,
    minimum width=1.6cm,
    minimum height=1cm,
    inner sep=1pt
  },
  dom/.style={
    -{Latex},
    densely dashed,
    shorten >=1pt,
    shorten <=1pt
  }
]
\node[blk] (C0) at ( 90:3.1) {$TT_{7}$}; \node[blk] (C1) at ( 30:3.1) {$TT_{3}$}; \node[blk] (C3) at (-30:3.1) {$TT_{3}$}; \node[blk] (C4) at (-90:3.1) {$TT_{9}$}; \node[blk] (C5) at (210:3.1) {$TT_{3}$}; \node[blk] (C2) at (150:3.1) {$TT_{11}$};
\node at ( 90:4.1) {$C_0$};   \node at ( 30:4.1) {$C_1$};
\node at (-30:4.1) {$C_3$};   \node at (-90:4.1) {$C_4$};
\node at (210:4.1) {$C_5$};   \node at (150:4.1) {$C_2$};
% the Hamiltonian cycle 0 -> 1 -> 3 -> 4 -> 5 -> 2 -> 0
\draw[dom,blue] (C0) to[bend left=12] (C1);
\draw[dom,blue] (C1) to[bend left=12] (C3);
\draw[dom,blue] (C3) to[bend left=12] (C4);
\draw[dom,blue] (C4) to[bend left=12] (C5);
\draw[dom,blue] (C5) to[bend left=12] (C2);
\draw[dom,blue] (C2) to[bend left=12] (C0);
% the remaining arcs
\draw[dom] (C3) to[bend left=10] (C0);
\draw[dom] (C1) to[bend right=10] (C4);
\draw[dom] (C5) to[bend left=10] (C3);
\draw[dom] (C4) to[bend right=10] (C2);
\draw[dom] (C0) to[bend left=10] (C5);
\draw[dom] (C2) to[bend right=10] (C1);
\draw[dom,red] (C0) to (C4);
\draw[dom,red] (C1) to (C5);
\draw[dom,red] (C2) to(C3);
\end{tikzpicture}
\end{center}
\caption{A tournament with no strong Seymour vertex, where the dashed arrow from
$C_i$ to $C_j$ means that $C_i$ dominates $C_j$.}
\label{fig:no-SSV}
\end{figure}
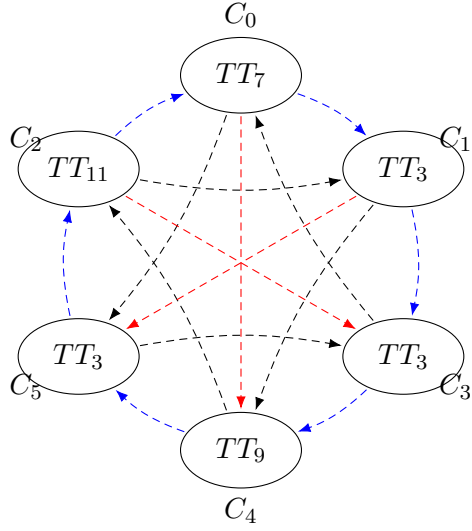

\begin{remark}[{David Dzitsoev}, personal communication]\label{rmk:no-SSV} \rm 
Let $Q$ be the tournament with $V(Q)=\{0,1,\ldots,5\}$ given by 
\[
\begin{array}{c|cccccc}
i & 0 & 1 & 2 & 3 & 4 & 5\\ \hline
N_Q^+(i) & \{1,4,5\} & \{3,4,5\} & \{0,1,3\} & \{0,4\} & \{2,5\} & \{2,3\}.
\end{array}
\]
Let $n_0,\ldots,n_5$ be positive integers and let $D$ be a \textit{blow-up} of
$Q$, that is,
$V(D)=\bigcup_{i=0}^{5}C_i$ for disjoint sets $C_0,\ldots,C_5$ with
$|C_i|=n_i$, where $D[C_i]$ is an empty graph and $C_i$ dominates
$C_j$ for every $j\in N_Q^+(i)$. 
Since 
\[
\begin{array}{c|cccccc}
i & 0 & 1 & 2 & 3 & 4 & 5\\ \hline
N_Q^{++}(i) & \{2,3\}   & \{0,2\}   & \{4,5\}   & \{1,2,5\} & \{0,1,3\} & \{0,1,4\}
\end{array}
\]
one can check that $D$ has no strong Seymour
vertex, whenever
\[
\begin{array}{lll}
n_1+n_4+n_5>n_2+n_3, &\quad n_4+n_5>n_2, &\quad n_0+n_1+n_3>n_4+n_5,\\[2pt]
n_0>n_1+n_5,         &\quad n_2+n_5>n_0+n_1+n_3, &\quad n_2>n_0+n_1 .
\end{array}
\]These conditions are satisfied, for instance, by
$(n_0,\ldots,n_5)=(7,3,11,3,9,3)$. 
It gives an oriented graph on $36$ vertices and
no strong Seymour vertex, disproving Conjecture~\ref{conj:SMC}.  
Moreover, taking these values and $D[C_i]=TT_{n_i}$
for every $i$ yields,  as shown in Figure~\ref{fig:no-SSV},  a tournament  with no strong Seymour
vertex, disproving Conjecture~\ref{conj:SMCinT}.
\end{remark}  Even in the classes where a strong Seymour vertex does exist, the two standard
candidates for it do not work, as the next two remarks show.

For a digraph $D$, 
a \textit{median order} of $D$ is an order $\tau: v_1,v_2,\ldots,v_{|V(D)|}$ of $V(D)$ such that the number $| \{ (v_i, v_j) \in A(D) : i < j \} |$ is maximized. 
The last vertex in a median order is called the \textit{feed vertex}. 
When $\tau$ is a median order, 
it holds that for every interval $I=[u, v]$ of $\tau$,
\begin{eqnarray} \label{eq:mo}
&& |N^+_D (u)\cap I| \geq|N^-_D (u)\cap I| \qquad \text{and}\qquad  | N^-_D (v)\cap I| \geq | N^+_D (v)\cap I|.
\end{eqnarray} 
An order $\tau$ of $V(D)$ is a \textit{local median order} 
if \eqref{eq:mo} holds for every interval $I=[u, v]$ of $\tau$.  
A \textit{tournament} is an orientation of a complete graph, and it has been proved in \cite{Havet2000} that the feed vertex (in any median order) of a tournament is a Seymour vertex.

\begin{figure}[ht]
\begin{center}
\begin{tikzpicture}[  x=1.0cm,y=1.0cm,every edge/.style={draw,postaction={decorate,decoration={markings,mark=at position 0.5 with {\arrow{latex}}}} }]
\draw [] (0,0) ellipse (0.6 and 1);
\node at (0,0) {$TT_{k+1}$};
\node at (0,-1.35) {$A_0$};
\node at (1,0) {$\rightarrow$};
\draw [] (2,0) ellipse (0.6 and 1);
\node at (2,0) {$TT_{k}$};
\node at (2,-1.35) {$A_1$};
\node at (3,0) {$\rightarrow$};
\draw [] (4,0) ellipse (0.6 and 1);
\node at (4,0) {$TT_k$};
\node at (4,-1.35) {$A_2$};
\node at (5,0) {$\rightarrow$};
\draw [] (6,0) ellipse (0.6 and 1.0);
\node at (6,0) {$TT_{k+2}$};
\node at (6,-1.35) {$A_3$};
\node at (7,0) {$\rightarrow$};
\draw [] (8,0) ellipse (0.6 and 1.0);
\draw[fill=black] (8,0) {coordinate (x)} circle (0.08);
\node at (8,-0.3) {$x$};
\node at (8,-1.35) {$A_4$};
\draw[->,densely dashed,blue] (0,1) .. controls (1,2) and (3,2) .. (3.8,1);
\draw[->,densely dashed,blue] (2,1) .. controls (3,1.8) and (5,1.8) .. (5.8,1);
\draw[->,densely dashed,blue] (4.2,1) .. controls (5,2) and (7,2) .. (7.8,1);
\draw[->,densely dashed,red] (8,-1) .. controls (6,-3) and (2,-3) .. (0.2,-1);
\draw[->,densely dashed,red] (8,-1) .. controls (6.5,-2) and (3.5,-2) .. (2.2,-1);
\draw[->,densely dashed,red] (6,-1) .. controls (4,-2) and (2,-2) .. (0.4,-1);
\end{tikzpicture}
\end{center}
\caption{A tournament in which the feed vertex of a median order is not a strong Seymour vertex, where the dotted arrow from $A_i$ to $A_j$ means that $A_i$ dominates $A_j$.}
\label{fig:FV-Not-SSV}
\end{figure}
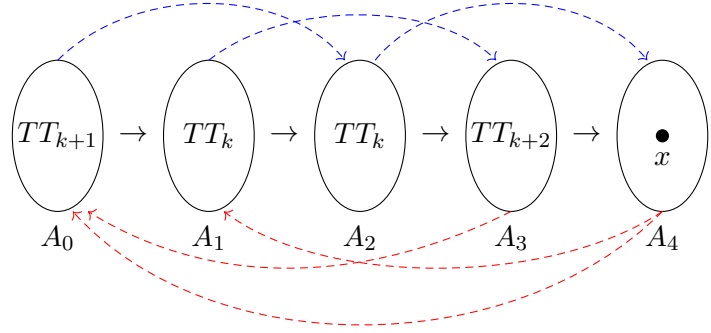

\begin{remark} \rm 
Let $k\geq 4$ be an integer. 
Let $T$ be a tournament, as shown in Figure~\ref{fig:FV-Not-SSV},  with $V(T)=\bigcup_{i=0}^4 A_i$
and $A_i$ dominates $A_{i+1}\cup A_{i+2}$, where $T[A_0]=TT_{k+1}$, $T[A_1]=T[A_2]=TT_{k}$, $A_3=TT_{k+2}$, $A_4=TT_1$, where the addition in the indices of $A_i$ is modulo $5$. If we list the vertices of each of $A_i$ in a median order, then we can obtain a median order of $T$. Moreover, let $A_4=\{x\}$. In addition, for every
median order of $T$, $x$ is the feed vertex and it is not a strong Seymour vertex, as one can see that there exists no complete matching from $A_0\cup A_1$ to $A_2\cup A_3$ in $T$. Thus
the construction of tournament $T$ indicates that a feed vertex may not be a strong Seymour vertex.
\end{remark}

It is natural to look for a (strong) Seymour vertex among the minimum out-degree vertices, but the idea fails even in tournaments.

\begin{remark} \rm 
When $\delta^+(D)\le 4$, there is a strong Seymour vertex in $D$ whose out-degree is $\delta^+(D)$ by Theorem~\ref{thm:SMC-MinOutdegree}.
If we look for a (strong) Seymour vertex which, in addition, has minimum out-degree,
then the minimum out-degree condition in Theorem~\ref{thm:SMC-MinOutdegree} is sharp.

\begin{figure}[ht]
\begin{center}
\begin{tikzpicture}[x=1.0cm,y=1.0cm,every edge/.style=
        {draw,postaction={decorate,decoration={markings,mark=at position 0.5 with {\arrow{latex}}}}}]
\draw[fill=black] (0,0) {coordinate (x)} circle (0.08);
\node at (-0.3,0) {$x$};
\node at (0.75,0){$\rightrightarrows$};
\draw [] (2,0) ellipse (0.5 and 1);
\node at (2,0) {$T^*_5$};
\node at (3,0) {$\rightrightarrows$};
\draw [] (4,0) ellipse (0.5 and 1);
\node at (4,0) {$T_4$};
\node at (5,0) {$\rightrightarrows$};
\draw [] (6,0) ellipse (0.5 and 1);
\node at (6,0) {$T^*_{13}$};
\draw[->,densely dashed,blue] (4,1) .. controls (3,2) and (0.5,2) .. (0,0.1);
\draw[->,densely dashed,red] (6,1) .. controls (5,2) and (3,2) .. (2.2,1);
\draw[->,densely dashed,red] (6,-1) .. controls (4,-2) and (1,-2) .. (0,-0.1);
\end{tikzpicture}
\end{center}
\caption{A tournament such that no Seymour vertex has minimum out-degree}
\label{fig:5-outdegree}
\end{figure}
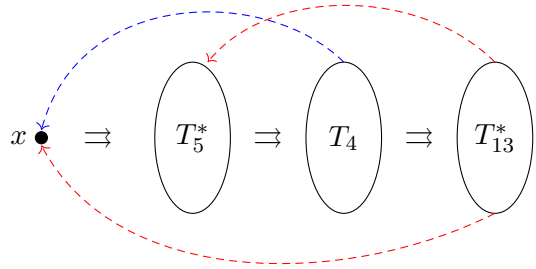

One can see that the tournament in Figure $\ref{fig:5-outdegree}$ has minimum out-degree $5$ but contains no desired (strong) Seymour vertex, here $T^*_{2k+1}$ refers to a $k$-regular tournament and $T_4$ refers to an arbitrary $4$-vertex tournament. Moreover, we can add all edges from $T^*_{13}$ to $T^*_5\cup \{x\}$, and from $T_4$ to $x$.
The obtained graph $T$ is a tournament with only one minimum out-degree vertex $x$, which is still not a Seymour vertex.
\end{remark}
%%%%%%%%%%%%%%
\section*{Acknowledgement}
%%%%%%%%%%%%%%
The authors are grateful to David Dzitsoev for  pointing out the construction in
Remark~\ref{rmk:no-SSV} using GPT.
The research of Yandong Bai and Binlong Li was supported in part by the National Key Research and Development Program of 
China (Grant No. 2026YFE0151700), National Natural Science Foundation of China (Grant Nos. 12542043, 12242111, 12131013), Guangdong Basic and Applied Basic Research Foundation (Grant No. 2023A1515030208), Shaanxi Fundamental Science Research Project for Mathematics and Physics (Grant No. 22JSZ009). 
Boram Park was supported by National Research Foundation of Korea (Grant No. RS-2025-00523206, NRF-2023K2A9A2A06059347) and the New Faculty Startup Fund from Seoul National University.


\begin{thebibliography}{19}

\bibitem{AGGWYZ2024}
J. Ai, S. Gerke, G. Gutin, S. Wang, A. Yeo, and Y. Zhou,
On Seymour's and Sullivan's second neighbourhood conjectures,
\textit{J. Graph Theory}, 105:413--426, 2024.

\bibitem{BM25}
Y. Bai and J. Ma,
Seymour's second neighborhood conjecture holds for $7$-anti-transitive oriented graphs,
in preparation.

\bibitem{BM08}
J. A. Bondy and U. S. R. Murty,
\textit{Graph Theory},
Graduate Texts in Mathematics 244, Springer, London, 2008.

\bibitem{BBKS2009}
J. Brantner, G. Brockman, B. Kay, and E. Snively,
Contributions to Seymour's second neighborhood conjecture,
\textit{Involve}, 2:387--395, 2009.

\bibitem{CC2023}
B. Chen and A. Chang,
A note on Seymour's second neighborhood conjecture,
\textit{Discrete Appl. Math.}, 337:272--277, 2023.

\bibitem{D20}
M. Daamouch,
Seymour's second neighborhood conjecture for $5$-anti-transitive oriented graphs,
\textit{Discrete Appl. Math.}, 285:454--457, 2020.

\bibitem{DFJN2022}
S. Dara, M. C. Francis, D. Jacob, and N. Narayanan,
Extending some results on the second neighborhood conjecture,
\textit{Discrete Appl. Math.}, 311:1--17, 2022.

\bibitem{DL95}
N. Dean and B. J. Latka,
Squaring the tournament---an open problem,
\textit{Congr. Numer.}, 109:73--80, 1995.

\bibitem{DGGK2024}
A. Espuny D\'{\i}az, A. Gir\~{a}o, B. Granet, and G. Kronenberg,
Seymour's second neighbourhood conjecture: random graphs and reductions,
\textit{Random Struct. Algorithms}, 66:e21251, 2025.

\bibitem{FY2007}
D. Fidler and R. Yuster,
Remarks on the second neighborhood problem,
\textit{J. Graph Theory}, 55:208--220, 2007.

\bibitem{Fisher1996}
D. Fisher,
Squaring a tournament: A proof of Dean's conjecture,
\textit{J. Graph Theory}, 23:43--48, 1996.

\bibitem{G2012}
S. Ghazal,
Seymour's second neighborhood conjecture for tournaments missing a generalized star,
\textit{J. Graph Theory}, 71:89--94, 2012.

\bibitem{H+21}
Z. R. Hassan, I. F. Khan, M. I. Poshni, and M. Shabbir,
Seymour's second neighborhood conjecture for $6$-antitransitive digraphs,
\textit{Discrete Appl. Math.}, 292:59--63, 2021.

\bibitem{Havet2000}
F. Havet and S. Thomass\'{e},
Median orders of tournaments: A tool for the second neighborhood problem and Sumner's conjecture,
\textit{J. Graph Theory}, 35:244--256, 2000.

\bibitem{HP24}
H. Huang and F. Peng,
An improved bound on Seymour's second neighborhood conjecture,
arXiv preprint, arXiv:2412.20234, 2024.

\bibitem{KanekoLocke2001}
Y. Kaneko and S. C. Locke,
The minimum degree approach for Paul Seymour's distance $2$ conjecture,
\textit{Congr. Numer.}, 148:201--206, 2001.

\bibitem{LX2017}
H. Liang and J.-M. Xu,
On Seymour's second neighborhood conjecture of $m$-free digraphs,
\textit{Discrete Math.}, 340:1944--1949, 2017.

\bibitem{Lim2020}
J. Lim,
A generalisation of Seymour's second neighbourhood conjecture,
arXiv preprint, arXiv:2001.07242, 2020.

\bibitem{Llado2013}
A. Llad\'{o},
On the second neighborhood conjecture of Seymour for regular digraphs with almost optimal connectivity,
\textit{Eur. J. Combin.}, 34:1406--1410, 2013.

\end{thebibliography}
\end{document}